\documentclass[11pt,reqno]{amsart}

\usepackage{amsbsy,amsfonts,amsmath,amssymb,amscd,amsthm,mathrsfs,verbatim,calc}
\usepackage{graphicx,epsfig,color}
\usepackage{enumitem}
\usepackage[a4paper,text={6in,7.5in},centering,includefoot,foot=1in]{geometry}
\usepackage[colorlinks=true,hypertexnames=false,linkcolor=blue,citecolor=green]{hyperref}

\small\normalsize
\theoremstyle{plain}
\newtheorem{mainthm}{Theorem}

\newtheorem{theorem}{Theorem}[section]

\newtheorem{proposition}[theorem]{Proposition}
\newtheorem{corollary}[theorem]{Corollary}

\newtheorem{definitions}[theorem]{Definitions}

\theoremstyle{definition}

\newtheorem{remark}[theorem]{Remark}

\numberwithin{equation}{section}
\let\oldmarginpar\marginpar
\renewcommand\marginpar[1]{\-\oldmarginpar[\raggedleft\footnotesize\textcolor{red}{#1}]{\raggedright\footnotesize\textcolor{red}{#1}}}

\newcommand{\Mf}{\mathcal{M}_f(M)}
\DeclareMathOperator{\acc}{acc}

\begin{document}

\title[Minimal Oscillation and Non-Statisticality]{Minimal Oscillation of Ces\`{a}ro Averages Implies Non-Statisticality}
\author[A. Sarizadeh]{Aliasghar Sarizadeh}
\address{Department of Mathematics, Ilam University, Ilam, Iran}
\email{ali.sarizadeh@gmail.com}
\email{a.sarizadeh@ilam.ac.ir}

\begin{abstract}

We investigate the relationship between the global convergence of Ces\`{a}ro averages and the pointwise statistical behavior of dynamical systems. 
First, we prove that if the Ces\`{a}ro averages accumulate on at least two different measures 
(a property we call the minimal oscillation property) then the system is non-statistical. 
Second, we show that a system possesses a natural measure in the strong sense if and only if it is uniquely ergodic. 
As a consequence, every minimal homeomorphism on the circle possesses a
natural measure in the strong sense which is physical and whose basin is the entire circle.

\end{abstract}

\maketitle

\section{Introduction: Preliminaries and Results}\label{sec:prelim}

Identifying topological conditions that guarantee statistical behavior or ergodicity remains a central challenge in dynamical systems. 
Minimality on compact metric spaces plays a foundational role in this study. 
Classical works of Oxtoby \cite{Ox52} and Furstenberg \cite{Fur81} established deep 
links between minimality, ergodicity, and recurrence phenomena. 
Nevertheless, it follows from results in \cite{F61,Ol} that the class of minimal 
homeomorphisms and the class of Lebesgue-ergodic homeomorphisms are incomparable.

Most known examples of non-statistical behavior (also called historic behavior) in the literature \cite{AA,BB23,CV,CYJ,Herman,HK0,HK,KST,T} are non-minimal. 
This motivates the following question:
\begin{quote}
Does there exist a minimal homeomorphism of a compact smooth manifold that is non-statistical with respect to Lebesgue measure?
\end{quote}

By basic topological dynamics, any minimal homeomorphism of the circle $S^1$ is topologically conjugate to an irrational rotation. Consequently, every minimal homeomorphism on $S^1$ is uniquely ergodic (see Proposition~\ref{prop:circle_minimal_UE} for a detailed proof). In particular, the set of non-statistical points is empty and such maps do not possess the minimal oscillation property of Cesàro averages.

Therefore, the circle $S^1$ cannot support a minimal non-statistical system. If the answer to the above question is affirmative, then any such minimal non-statistical homeomorphism must reside on a smooth manifold of dimension at least two.

In this note we introduce an oscillation property and prove two main results. 
The first links oscillation of Ces\`{a}ro averages to non-statisticality. 
The second characterizes strong naturality in terms of unique ergodicity. 
As an application, we show that every minimal homeomorphism on the circle possesses a unique physical measure whose basin is the entire circle.

We now recall the basic notation. Let $M$ be a compact, connected, boundaryless smooth manifold equipped with its normalized Lebesgue volume measure $m$. Let $\mathcal{M}(M)$ be the space of Borel probability measures on $M$, and let $\mathcal{M}_f(M)$ be the subset of $f$-invariant measures. Both spaces are equipped with the weak* topology.

For a point $x \in M$, the empirical measures along its orbit are defined by
\[
e_n^f(x) := \frac{1}{n} \sum_{i=0}^{n-1} \delta_{f^i(x)},
\]
where $\delta_y$ denotes the Dirac measure at $y$.

In cases where the sequence $\{e_n^f(x)\}_{n \in \mathbb{N}}$ converges in the weak* topology, we denote its limit by $e^f_\infty(x)$. Equivalently, $e^f_\infty(x)$ is characterized by the property that for every continuous function $\phi : M \to \mathbb{R}$,
\[
\lim_{n \to \infty} \frac{1}{n} \sum_{j=0}^{n-1} \phi(f^j(x)) = \int_M \phi \, d e^f_\infty(x).
\]
We define the set of \emph{statistical points} as
\[
\mathcal{SP}(f) = \{ x \in M : e_n^f(x) \text{ converges in the weak* topology} \},
\]
and let $\mathcal{NSP}(f) = M \setminus \mathcal{SP}(f)$ denote the set of non-statistical points. 
Points belonging to $\mathcal{NSP}(f)$ are said to exhibit \emph{historic behavior}: their empirical measures $\{e_n^f(x)\}$ do not converge and keep oscillating between different accumulation points in $\mathcal{M}(M)$ as $n\to\infty$.

The system $f$ is called \emph{statistical} if $m(\mathcal{NSP}(f)) = 0$. Otherwise, if $m(\mathcal{NSP}(f)) > 0$, the system is called \emph{non-statistical} (or said to exhibit \emph{historic behavior} on a set of positive Lebesgue measure).

The Ces\`{a}ro averages of Lebesgue measure are defined by
\[
\mu_n = \frac{1}{n} \sum_{i=0}^{n-1} f_*^i m.
\]
We denote by $\acc(\nu_n)$ the set of accumulation points of a sequence of measures $\nu_n$.

\begin{definitions}
The system $f$ is said to possess:
\begin{enumerate}
    \item[(i)]    
    the \emph{minimal oscillation property of Ces\`{a}ro averages} if
    \[
    \# \acc\left( \frac{1}{n} \sum_{i=0}^{n-1} f_*^i m\right) > 1;
    \]
    \item[(ii)] a \emph{natural measure in the strong sense} $\mu$ if for every probability measure $\nu\ll m$,
    \[
    \frac{1}{n}\sum_{i=0}^{n-1} f_*^i \nu  \xrightarrow{w^*} \mu \quad \text{as } n\to\infty;
    \]
     \item[(iii)] a \emph{natural measure in the weak sense} $\mu$ if 
    \[
    \frac{1}{n}\sum_{i=0}^{n-1} f_*^i m  \xrightarrow{w^*} \mu \quad \text{as } n\to\infty.
    \]
\end{enumerate}
\end{definitions}
The relationship between the global convergence of Ces\`{a}ro averages and pointwise 
statistical behavior lies at the heart of Theorem~\ref{bipolar}. In fact, the minimal oscillation 
property of Ces\`{a}ro averages provides a simple criterion for identifying non-statistical systems.

\begin{mainthm}\label{bipolar}
If $f$ possesses the minimal oscillation property of Ces\`{a}ro averages 
with respect to Lebesgue measure $m$, then $f$ is non-statistical with respect to $m$.
\end{mainthm}
 By the contrapositive of Theorem~\ref{bipolar}, the convergence of the Cesàro averages of Lebesgue measure is a necessary condition for the system to be statistical. 

This naturally raises the question of sufficiency: Does naturality  in weak sense  guarantee that $m(\mathcal{NSP}(f)) = 0$? 
It is known that the existence of an ergodic natural measure (original definition) does not necessarily imply that it is physical,  \cite{BB03, JT}. 
Therefore, determining under which additional conditions naturality in weak sense implies statisticality remains an interesting open problem.
Theorem~\ref{natural} below provides a condition that trivially guarantees statistical behavior.

\begin{mainthm}\label{natural}
A dynamical system $f$ possesses a natural measure in the strong sense if and only if $f$ is uniquely ergodic. 
In this case, the natural measure $\mu$ is the unique $f$-invariant probability measure, $e_n^f(x) \to \mu$ for all $x \in M$ (meaning $B(\mu) = M$), and $\mu$ is ergodic and physical.
\end{mainthm}

\begin{mainthm}\label{maincor}
If $f:S^1\to S^1$ is a minimal homeomorphism, then it possesses a natural measure in the strong sense which is physical.
\end{mainthm}

\section{The Minimal Oscillation Property: A Non-Statistical Criterion}\label{sec:oscillation}

In this section, we prove that any system satisfying the minimal oscillation property of Ces\`{a}ro averages must be non-statistical. 
The argument relies on showing that the inherent oscillation prevents the Birkhoff averages from converging on a set of full measure.
In \cite{AA}, the authors prove a stronger version, showing that the maximal oscillation property, $\acc\left( \frac{1}{n} \sum_{i=0}^{n-1} f_*^i m\right) \supseteq \mathcal{M}_f(M)$, implies that the set of non-statistical points has full measure.

\begin{proof}[Proof of Theorem~\ref{bipolar}]
Let $\mathcal{SP} = M \setminus \mathcal{NSP}$; that is, the set of points $x$ such that the Birkhoff limit $e^f_\infty(x) = \lim_{n \to \infty} \frac{1}{n} \sum_{k=0}^{n-1} \delta_{f^k(x)}$ is well-defined. 
If $m(\mathcal{SP}) = 0$, then trivially $m(\mathcal{NSP}) = 1 > 0$, and there is nothing to prove. 

Assume instead, for the sake of contradiction, that $\mathcal{SP}$ has full Lebesgue measure, meaning $m(\mathcal{SP}) = 1$. 
In this case, the normalized restriction $m_{\scriptscriptstyle \mathcal{SP}}$ is exactly equal to $m$.

Consider the sequence of Ces\`{a}ro averages for the measure $m_{\scriptscriptstyle \mathcal{SP}}$:
\[
\mu_n = \frac{1}{n} \sum_{k=0}^{n-1} f^k_* m_{\scriptscriptstyle \mathcal{SP}}.
\]
We claim that this sequence converges to a unique limit measure in the weak-* topology. 
To see this, let $\phi : M \to \mathbb{R}$ be an arbitrary continuous function. Observe that
\[
\int \phi \, d\mu_n = \int \phi \, d\left( \frac{1}{n} \sum_{k=0}^{n-1} f^k_* m_{\scriptscriptstyle \mathcal{SP}} \right) = \int \left( \frac{1}{n} \sum_{k=0}^{n-1} \phi(f^k(x)) \right) dm_{\scriptscriptstyle \mathcal{SP}}(x).
\]
Since $m_{\scriptscriptstyle \mathcal{SP}}(\mathcal{SP}) = 1$, for $m_{\scriptscriptstyle \mathcal{SP}}$-almost every $x \in \mathcal{SP}$, the integrand $\frac{1}{n} \sum_{k=0}^{n-1} \phi(f^k(x))$ converges pointwise to $\int \phi \, d[e^f_\infty(x)]$ by the definition of the Birkhoff limit. 
Because $\phi$ is bounded on the compact manifold $M$, we may apply the Dominated Convergence Theorem to interchange the limit and the integral, yielding
\[
\lim_{n \to \infty} \int \phi \, d\mu_n = \int \left( \int \phi \, d[e^f_\infty(x)] \right) dm_{\scriptscriptstyle \mathcal{SP}}(x).
\]
Therefore, there exists a limiting measure $\mu_{\scriptscriptstyle \mathcal{SP}}$ such that $\int \phi \, d\mu_{\scriptscriptstyle \mathcal{SP}} = \lim_{n \to \infty} \int \phi \, d\mu_n$.
Because we assumed $m(\mathcal{SP}) = 1$, we have $m_{\scriptscriptstyle \mathcal{SP}} = m$. 
The calculation above therefore demonstrates that the sequence $\{\frac{1}{n} \sum_{k=0}^{n-1} f^k_* m \}$ converges to the unique measure $\mu_{\scriptscriptstyle \mathcal{SP}}$. 
This directly contradicts the hypothesis that $f$ possesses the minimal oscillation property of Ces\`{a}ro averages, which requires $\{\frac{1}{n} \sum_{k=0}^{n-1} f^k_* m\}$ to have at least two distinct accumulation points, thereby precluding convergence. 

Consequently, our assumption that $\mathcal{SP}$ has full measure must be false. 
It follows that $m(\mathcal{NSP}) > 0$, which, by definition, means that $f$ is non-statistical with respect to $m$.
\end{proof}

\begin{corollary}\label{SP}
Let $m(\mathcal{SP}) > 0$ and let $m_{\scriptscriptstyle \mathcal{SP}}$ denote the normalized restriction of the Lebesgue measure to $\mathcal{SP}$.
Then the sequence $\{ \frac{1}{n} \sum_{k=0}^{n-1} f^k_* m_{\scriptscriptstyle \mathcal{SP}}\}_n$ converges in the weak* topology.
\end{corollary}
\begin{proof}
The proof follows exactly the same argument as the proof of Theorem~\ref{bipolar}. 
For any continuous function $\phi$, the Birkhoff average $\frac{1}{n}\sum_{k=0}^{n-1} \phi(f^k(x))$ converges for all $x \in \mathcal{SP}$ by definition. 
Since $m_{\scriptscriptstyle \mathcal{SP}}$ is supported entirely on $\mathcal{SP}$, this pointwise convergence holds $m_{\scriptscriptstyle \mathcal{SP}}$-almost everywhere. 
Applying the Dominated Convergence Theorem (as $\phi$ is bounded on the compact manifold $M$) allows us to interchange the limit and the integral, proving that the sequence of measures $\frac{1}{n} \sum_{k=0}^{n-1} f^k_* m_{\scriptscriptstyle \mathcal{SP}}$ converges in the weak* topology.
\end{proof}

\begin{remark}\label{rem:maximal}
A map $f : M \to M$ possesses the \emph{maximal oscillation property of Ces\`{a}ro averages} with respect to $m$ if the sequence of averages 
\[
m_j=\frac{1}{j}\sum_{i=0}^{j-1}f_*^im 
\]
accumulates on the entire space of $f$-invariant measures (i.e., $\acc(m_j) \supseteq \Mf$). 
Note that maximal oscillation does not imply minimal oscillation. 
For instance, a uniquely ergodic system trivially satisfies maximal oscillation (since $\Mf$ is a singleton) but fails to exhibit minimal oscillation, as its Ces\`{a}ro averages converge to a single measure.
\end{remark}

\begin{remark}\label{lem:SP_full}
Let $f$ possess a natural measure $\mu$ in the weak sense, i.e., 
 $\frac{1}{n} \sum_{k=0}^{n-1} f_*^k m \to \mu$ as $n\to\infty$. 
Then $\mu(\mathcal{SP}) = 1$.

Indee, a standard property of weak-* limits of Ces\`{a}ro averages guarantees that $\mu$ is $f$-invariant. 
By Birkhoff's Ergodic Theorem, for $\mu$-almost every $x \in M$, the empirical measures $e_n^f(x)$ converge to some invariant measure. 
This implies that $\mu$-almost every point belongs to the statistical set $\mathcal{SP}$, giving us $\mu(\mathcal{SP}) = 1$.
\end{remark}

\section{Strong Naturality and Unique Ergodicity}\label{sec:natural}
The converse of Theorem~\ref{bipolar} does not hold. Misiurewicz \cite{Mis} constructed a continuous map 
 $f$ on the torus $\mathbb{T}^2$ that possesses a natural measure in the weak sense 
(guaranteeing the convergence of its Ces\`{a}ro averages) yet for Lebesgue almost every $x$, the empirical measures $e_n^f(x)$ exhibit maximal oscillation, accumulating on the entirety of $\mathcal{M}_f(M)$.

 Theorem~\ref{natural} establishes an equivalence between strong naturality and unique ergodicity, and 
 so for every $x\in M$, the empirical measures $e_n^f(x)$  convergent to unique limit measure.

\begin{proof}[Proof of Theorem~\ref{natural}]
We prove both directions of the equivalence. 

Firstly, assume $f$ is uniquely ergodic with unique invariant probability measure $\mu$. 
By classical results in topological dynamics (Oxtoby's theorem), unique ergodicity implies that for \emph{every} point $x \in M$, the empirical measures converge: $e_n^f(x) \xrightarrow{w^*} \mu$. 
Let $\nu$ be any probability measure absolutely continuous with respect to $m$ ($\nu \ll m$), and let $\phi : M \to \mathbb{R}$ be an arbitrary continuous function. Since $M$ is compact, $\phi$ is bounded. 
Because $\frac{1}{n}\sum_{k=0}^{n-1} \phi(f^k(x)) \to \int \phi \, d\mu$ for every $x \in M$, we can apply the Dominated Convergence Theorem to interchange the limit and the integral with respect to $\nu$:
\begin{align*}
\lim_{n \to \infty} \int \phi \, d\left( \frac{1}{n}\sum_{k=0}^{n-1} f_*^k \nu \right) 
&= \lim_{n \to \infty} \int \left( \frac{1}{n}\sum_{k=0}^{n-1} \phi(f^k(x)) \right) d\nu(x) \\[4pt]
&= \int \lim_{n \to \infty} \left( \frac{1}{n}\sum_{k=0}^{n-1} \phi(f^k(x)) \right) d\nu(x) \\[4pt]
&= \int \left( \int \phi \, d\mu \right) d\nu(x) = \int \phi \, d\mu.
\end{align*}
Since $\phi$ was arbitrary, $\frac{1}{n}\sum_{k=0}^{n-1} f_*^k \nu \xrightarrow{w^*} \mu$. Because $\nu \ll m$ was arbitrary, $\mu$ is a natural measure in the strong sense.

Converse, let $d(\cdot, \cdot)$ be a metric inducing the weak* topology on $\mathcal{M}(M)$. Define the set of measures whose Ces\`{a}ro averages converge to $\mu$:
\[
S_\mu := \left\{ \nu \in \mathcal{M}(M) : \frac{1}{n}\sum_{i=0}^{n-1} f_*^i \nu \to \mu \right\}.
\]
By hypothesis, $S_\mu$ contains all measures absolutely continuous with respect to $m$. We prove $S_\mu = \mathcal{M}(M)$ in two steps: showing $S_\mu$ is closed, and showing absolutely continuous measures are dense.

\emph{$S_\mu$ is closed.} Let $\nu_k \in S_\mu$ be a sequence converging to some $\nu \in \mathcal{M}(M)$. Suppose for contradiction that $\nu \notin S_\mu$. Then there exists a continuous function $\phi: M \to \mathbb{R}$, some $\epsilon > 0$, and a subsequence $n_j \to \infty$ such that 
\[
\left| \int \phi \, d\left( \frac{1}{n_j}\sum_{i=0}^{n_j-1} f_*^i \nu \right) - \int \phi \, d\mu \right| \ge \epsilon \quad \text{for all } j.
\]
Since $\nu_k \in S_\mu$, for each $k \in \mathbb{N}$ there exists $t_k$ such that 
 $\left| \int \phi \, d\left( \frac{1}{n}\sum_{i=0}^{n-1} f_*^i \nu_k \right) - \int \phi \, d\mu \right| < \frac{\epsilon}{2}$ for all $n \geq t_k$. 

Because the pushforward operator $f_*$ is continuous in the weak* topology, the map $\nu \mapsto \frac{1}{n}\sum_{i=0}^{n-1} f_*^i \nu$ is continuous for any fixed $n$. 
Thus, for sufficiently large $k$ and any $n_j \geq t_k$, we have 
\[
\left| \int \phi \, d\left( \frac{1}{n_j}\sum_{i=0}^{n_j-1} f_*^i \nu_k \right) - \int \phi \, d\left( \frac{1}{n_j}\sum_{i=0}^{n_j-1} f_*^i \nu \right) \right| < \frac{\epsilon}{2}.
\]
Applying the triangle inequality with the two estimates above yields a contradiction. Therefore, $\nu \in S_\mu$, and $S_\mu$ is closed.

\emph{Absolutely continuous measures are dense.} 
Because $M$ is a connected, boundaryless smooth manifold and $m$ is a volume probability, the support of $m$ is all of $M$. 
Any Borel probability measure $\nu$ on a compact Riemannian manifold can be weakly approximated by a sequence of smooth measures (see Villani \cite{Vill}). Any smooth measure on $M$ admits a continuous density with respect to $m$, meaning it is absolutely continuous with respect to $m$. 

Since $S_\mu$ is a closed set containing a dense subset (the absolutely continuous measures), it follows that $S_\mu = \mathcal{M}(M)$. 
In particular, for every point $x \in M$, the Dirac measure $\delta_x$ belongs to $S_\mu$. Observing that 
 $\frac{1}{n}\sum_{i=0}^{n-1} f_*^i \delta_x = e_n^f(x)$, we immediately obtain 
\[
e_n^f(x) \to \mu \quad \text{for all } x \in M.
\]
This implies $B(\mu) = M$. Because the empirical measures for all points converge to the same measure $\mu$, there can be no other invariant probability measures. Indeed,  assume there exists another $f$-invariant probability measure $\nu \in \mathcal{M}_f(M)$ with $\nu \neq \mu$.

Let $\phi: M \to \mathbb{R}$ be an arbitrary continuous function. By assumption,
\[
\frac{1}{n}\sum_{k=0}^{n-1} \phi(f^k(x)) \longrightarrow \int_M \phi \, d\mu \quad \text{for every } x \in M.
\]

Integrate both sides with respect to $\nu$:
\[
\int_M \left( \lim_{n\to\infty} \frac{1}{n}\sum_{k=0}^{n-1} \phi(f^k(x)) \right) d\nu(x) = \int_M \phi \, d\mu.
\]

Since $\phi$ is bounded (as $M$ is compact), the Dominated Convergence Theorem allows us to pass the limit inside:
\[
\lim_{n\to\infty} \int_M \frac{1}{n}\sum_{k=0}^{n-1} \phi(f^k(x)) \, d\nu(x) = \int_M \phi \, d\mu.
\]

The left-hand side simplifies as follows:
\[
\lim_{n\to\infty} \frac{1}{n} \sum_{k=0}^{n-1} \int_M \phi(y) \, d(f_*^k \nu)(y).
\]
Because $\nu$ is $f$-invariant, $f_*^k\nu = \nu$ for all $k$, so each term equals $\int_M \phi \, d\nu$. Hence the average is constantly $\int_M \phi \, d\nu$, and we obtain
\[
\int_M \phi \, d\nu = \int_M \phi \, d\mu.
\]

Since $\phi$ was arbitrary, $\mu = \nu$ in the weak* topology, contradicting $\nu \neq \mu$.

Therefore, $\mu$ is the unique $f$-invariant probability measure, and so $f$ is uniquely ergodic.
\end{proof}

\section{The Limit of Cesàro Averages for Minimal Homeomorphisms on the Circle is Physical}
\label{sec:circle_physical}

We now turn to circle homeomorphisms. The following proposition establishes that every minimal homeomorphism of $S^1$ is uniquely ergodic.

\begin{proposition}\label{prop:circle_minimal_UE}
Let $f: S^1 \to S^1$ be a minimal homeomorphism. Then $f$ is uniquely ergodic.
\end{proposition}

\begin{proof}
Since $f$ is minimal, it has no periodic points. By a classical result (see, e.g., Katok and Hasselblatt \cite{KH}, Problem~11.2.4), every orientation-reversing homeomorphism of $S^1$ has exactly two fixed points. Therefore, minimality forces $f$ to be orientation-preserving.

As $f$ is an orientation-preserving circle homeomorphism, the Poincar\'e rotation number $\rho(f)$ is well-defined. Moreover, $\rho(f)$ is rational if and only if $f$ has a periodic point. Since $f$ has no periodic points, $\rho(f) = \alpha$ is irrational.
By the classical theory of circle homeomorphisms (Poincaré's theorem), there exists a continuous, non-decreasing, surjective map 
 $h: S^1 \to S^1$ such that
\[
h \circ f = R_\alpha \circ h,
\]
where $R_\alpha$ denotes rotation by angle $\alpha$.

We claim that minimality of $f$ implies that this semi-conjugacy is actually a topological conjugacy (i.e., $h$ is a homeomorphism). Suppose, for contradiction, that there exists a wandering interval $J \subset S^1$, that is, the forward iterates $\{f^n(J)\}_{n \geq 0}$ are pairwise disjoint. Since the total length of $S^1$ is $1$, we have $|f^n(J)| \to 0$ as $n \to \infty$. For any $x \in J$, the $\omega$-limit set satisfies
\[
\omega(x) \subseteq \bigcap_{N \geq 0} \overline{\bigcup_{n \geq N} f^n(J)}.
\]
For sufficiently large $N$, the set on the right is a proper closed subset of $S^1$. 
Thus $\omega(x) \neq S^1$, contradicting minimality. 
Therefore, $f$ admits no wandering intervals, so $h$ is injective and hence a homeomorphism. It follows that $f$ is topologically conjugate to $R_\alpha$.

Since the irrational rotation $R_\alpha$ is uniquely ergodic (with unique invariant probability measure equal to normalized Lebesgue measure $m$), and unique ergodicity is preserved under topological conjugacy, we conclude that $f$ is uniquely ergodic, with unique invariant probability measure $\mu = h_* m$.
\end{proof}

\begin{proof}[Proof of Theorem~\ref{maincor}]
Let $f: S^1 \to S^1$ be a minimal homeomorphism. By Proposition~\ref{prop:circle_minimal_UE}, $f$ is uniquely ergodic. By Theorem~\ref{natural}, $f$ therefore possesses a natural measure $\mu$ in the strong sense, and $\mu$ is physical with $B(\mu) = S^1$.
\end{proof}
\subsection*{acknowledgments}
The author would like to thank Abbas Fakhari for pointing out that no homeomorphism of the circle can possess the minimal oscillation property of Cesàro averages. This observation helped improve the presentation of the results concerning circle homeomorphisms.

\end{document}